\documentclass[12pt]{amsart}

\usepackage{amssymb, amsmath}
\usepackage{amsthm}
\usepackage{mathrsfs}
\usepackage{cite}
\usepackage{enumitem}
\usepackage{cases}
\usepackage[font={it}, margin=1cm]{caption}
\usepackage{verbatim}

\newtheorem{theorem}{Theorem}[section]
\newtheorem{question}[theorem]{Question}
\newtheorem{lemma}[theorem]{Lemma} 
\newtheorem{proposition}[theorem]{Proposition} 
\newtheorem{thmletter}{Theorem}

\newtheorem{corollary}[theorem]{Corollary} 

\newtheoremstyle{example}{3pt}{3pt}{}{10pt}{\itshape}{:}{.5em}{}
\theoremstyle{example}

\newcommand{\p}[1]{\noindent {\newline\bf #1.}}

\newcommand{\aut}{\operatorname{Aut}}

\title[Residual finiteness of certain HNN-extensions]
{The residual finiteness of (hyperbolic) automorphism-induced HNN-extensions}
\author{Alan D. Logan}%
\date{\today}
\address{School of Mathematics and Statistics\\ University of Glasgow\\ Glasgow, G12 8QW, UK.}
\email{Alan.Logan@glasgow.ac.uk}


\begin{document}
\maketitle

\begin{abstract}
We classify finitely generated, residually finite automorphism-induced HNN-extensions in terms of the residual separability of a single associated subgroup. This classification provides a method to construct automorphism-induced HNN-extensions which are not residually finite. We prove that this method can never yield a ``new'' counter-example to Gromov's conjecture on the residual finiteness of hyperbolic groups.
\end{abstract}

\section{Introduction}
\label{introduction}



A group $H\ast_{(K, \phi)}$ is called an \emph{automorphism-induced HNN-extension} if it has a relative presentation of the form
\[
H\ast_{(K, \phi)}=\langle H, t; tkt^{-1}=\phi(k), k\in K\rangle
\]
where $\phi\in\aut(H)$ and $K\lneq H$.

The main result of this note is a classification of finitely generated, residually finite automorphism-induced HNN-extensions. A subgroup $K$ of $H$ is \emph{residually separable} in $H$ if for all $x\in H\setminus K$ there exists a finite index, normal subgroup $N$ of $H$, written $N\lhd_f H$, such that $x\not\in KN$ (hence if $\varphi_x: H\rightarrow H/N$ is the natural map then $\varphi_x(x)\not\in\varphi_x(K)$).

\begin{thmletter}
\label{thm:RF}
Suppose that $H$ is finitely generated. Then $G=H\ast_{(K, \phi)}$ is residually finite if and only if $H$ is residually finite and $K$ is residually separable in $H$.
\end{thmletter}

We prove two corollaries of Theorem~\ref{thm:RF}. These corollaries can be easily applied to construct automorphism-induced HNN-extensions which are not residually finite. Both corollaries relate to the subgroup-quotient $N_H(K)/K$. This subgroup-quotient plays a central role in a framework for the construction of groups possessing certain properties and with specified outer automorphism group \cite{logan2015Bass} (see also \cite{LoganNonRecursive} \cite{LoganHNN}).

\begin{corollary}
\label{corol:subgroupQuotient}
Suppose that $H$ is finitely generated. If $N_H(K)/K$ is not residually finite then $G=H\ast_{(K, \phi)}$ is not residually finite.
\end{corollary}


\begin{corollary}
\label{corol:subgroupQuotient+FI}
Suppose that $H$ is finitely generated and that $N_H(K)$ has finite index in $H$. Then $G=H\ast_{(K, \phi)}$ is residually finite if and only if both $H$ and $N_H(K)/K$ are residually finite.
\end{corollary}


\p{Hyperbolicity}
It is a famous conjecture of Gromov that all hyperbolic groups are residually finite \cite{niblo1991problems} \cite{kapovich2000equivalence} \cite{Olshanskii2000BassLubotzky}. One might hope to apply Corollary \ref{corol:subgroupQuotient} to obtain a counter-example to this conjecture.
However, Theorem \ref{thm:hyperbolicity} proves that Corollary \ref{corol:subgroupQuotient} can produce no ``new'' counter-examples to Gromov's conjecture, in the sense that if $G=H\ast_{(K, \phi)}$ is a counter-example where the subgroup-quotient $N_H(K)/K$ is used to force $G$ to be non-residually finite then the conditions of Theorem \ref{thm:hyperbolicity} hold, and so $H$ is also a counter-example.

\begin{thmletter}
\label{thm:hyperbolicity}
Suppose that $G=H\ast_{(K, \phi)}$ is hyperbolic and non-residually finite, and that $K\lneq N_H(K)$. Then $K$ is finite, and $H$ is hyperbolic and non-residually finite.
\end{thmletter}

Theorem \ref{thm:hyperbolicity} leaves the following question:

\begin{question}
Suppose that $G=H\ast_{(K, \phi)}$ is hyperbolic and non-residually finite. Then is $H$ is hyperbolic and non-residually finite?
\end{question}

We also have the following result:

\begin{thmletter}
\label{thm:hyperbolicityZxZ}
Suppose that $K\lneq N_H(K)$ and that $K$ contains an element of infinite order. Then $\mathbb{Z}\times\mathbb{Z}$ embeds into $G=H\ast_{(K, \phi)}$.
\end{thmletter}

Automorphism-induced HNN-extensions can be thought of as ``partial'' mapping tori $H\rtimes_{\phi}\mathbb{Z}$. Theorem \ref{thm:hyperbolicityZxZ} proves that automorphism-induced HNN-extensions of free groups $F_n\ast_{(K, \phi)}$ are not hyperbolic if $K\lneq N_H(K)$, even if the ``full'' mapping torus $F_n\rtimes_{\phi}\mathbb{Z}$ is hyperbolic.

\p{Acknowledgments}
I would like to thank the anonymous referee for their extremely helpful comments.


\section{Residual finiteness}
\label{sec:RFproofs}
We first prove Theorem~\ref{thm:RF}.
Note that for $G$ some group, if $P\lhd_f G$ and $H\leq G$ then $P\cap H\lhd_fH$.
Also note that if $H$ is a finitely generated group, $Q\lhd_fH$ and $\phi\in\aut(H)$ then $\cap_{i\in \mathbb{Z}}\phi^i(Q)\lhd_fH$.

\begin{proof}[Proof of Theorem~\ref{thm:RF}]
Suppose $H$ is residually finite and $K$ is residually separable in $H$. Then $H\ast_{(K, \phi)}$ is residually finite \cite[Lemma 4.4]{BaumslagTretkoff}.

Suppose $H\ast_{(K, \phi)}$ is residually finite. Then $H$ is residually finite, as subgroups of residually finite groups are residually finite. Now, suppose that $K$ is not residually separable in $H$, and let $x\in H\setminus K$ be such that $x\in KN$ for all finite index subgroups $N$ of $H$. Let $\overline{N}\lhd_f H\ast_{(K, \phi)}$ be arbitrary. It is sufficient to prove that $txt^{-1}\phi(x)^{-1}\in \overline{N}$. To see this inclusion, first note that $\overline{N}\cap H\lhd_fH$. Consider $L:=\cap_{i\in\mathbb{Z}}\phi^i\left(\overline{N}\cap H\right)$, and note that $L\lhd_fH$. Then there exists $k\in K$ such that $xk^{-1}, \phi(xk^{-1})\in L$. Thus, $xk^{-1}, \phi(xk^{-1})\in \overline{N}$, and so $x\overline{N}=k\overline{N}$ and $\phi(x)\overline{N}=\phi(k)\overline{N}$. Then:
\[
txt^{-1}\phi(x)^{-1}\overline{N}
=tkt^{-1}\phi(k)^{-1}\overline{N}
=\overline{N}
\]
Hence, $txt^{-1}\phi(x)^{-1}\in\overline{N}$ as required.
\end{proof}

We now prove Corollary~\ref{corol:subgroupQuotient}.

\begin{proof}[Proof of Corollary~\ref{corol:subgroupQuotient}]
Suppose that $N_H(K)/K$ is not residually finite. Then there exists some $x\in N_H(K)$ such that $x\in NK$ for all $N\lhd_fN_H(K)$. Hence, for all $\overline{N}\lhd_f H$ we have that $x\in \left(\overline{N}\cap N_H(K)\right)K$, and so $x\in\overline{N}K$. Therefore, $K$ is not residually separable in $H$, and so $H\ast_{(K, \phi)}$ is not residually finite by Theorem~\ref{thm:RF}.
\end{proof}

We now prove Corollary~\ref{corol:subgroupQuotient+FI}.
We previously proved the analogous result for the groups $H\ast_{(K, 1)}$, so where the inducing automorphism $\phi$ is trivial \cite[Proposition 2.2 ]{LoganNonRecursive}.

\begin{proof}[Proof of Corollary~\ref{corol:subgroupQuotient+FI}]
By Theorem \ref{thm:RF} and Corollary \ref{corol:subgroupQuotient}, it is sufficient to prove that if $H$ and $N_H(K)/K$ are residually finite then $K$ is residually separable. So, suppose that $H$ and $N_H(K)/K$ are residually finite.

Additionally, suppose that $x\not\in N_H(K)$. Clearly $x\not\in N_H(K)K$ as $K\leq N_H(K)$. Then the subgroup $N:=\cap_{h\in H} h^{-1}N_H(K)h$ is a finite index, normal subgroup of $H$ such that $x\not\in NK$, as required.

Suppose that $x\in N_H(K)\setminus K$. Now, as $N_H(K)/K$ is residually finite, there exists a map $\varphi_x: N_H(K)/K\rightarrow A_x$ with $A_x$ finite and $xK\not\in\ker(\varphi_x)$. Therefore, there exists a map $\widetilde{\varphi_x}: N_H(K)\rightarrow A_x$ which factors as $N_H(K)\rightarrow N_H(K)/K\xrightarrow{\varphi_x} A_x$ such that $x\not\in\ker\left(\widetilde{\varphi_x}\right)$.
Then $K\leq \ker\left(\widetilde{\varphi_x}\right)$ so $x\not\in \ker\left(\widetilde{\varphi_x}\right)K$.
As $\ker\left(\widetilde{\varphi_x}\right)\lhd_fN_H(K)\lhd_fH$, there exists $N\lhd_fH$ such that $N\leq \ker\left(\widetilde{\varphi_x}\right)$. As $x\not\in \ker\left(\widetilde{\varphi_x}\right)K$ and $NK\leq \ker\left(\widetilde{\varphi_x}\right)K$ we have that $x\not\in NK$ as required.
\end{proof}


\section{Hyperbolicity}
\label{sec:Hypproofs}
%
%
%
%
%
%
%
%
We first prove Theorem \ref{thm:hyperbolicityZxZ}, as it is applied in the proof of Theorem \ref{thm:hyperbolicity}.
Recall that Theorem \ref{thm:hyperbolicityZxZ} gives a necessary condition for $\mathbb{Z}\times\mathbb{Z}$ to embed into $G=H\ast_{(K, \phi)}$. As $\mathbb{Z}\times\mathbb{Z}$ does not embed into any hyperbolic group, Theorem \ref{thm:hyperbolicityZxZ} gives a necessary condition for the hyperbolicity of automorphism-induced HNN-extensions.

\begin{proof}[Proof of Theorem \ref{thm:hyperbolicityZxZ}]
Consider an element $k\in K$ of infinite order, and consider $a\in N_H(K)\setminus K$. Then the word $W=a^{-1}t^{-1}\phi(a)t$ has infinite order in $G$, and indeed no power of $W$ is contained in $K$. Now, as $aka^{-1}\in K$ we have that $t^{-1}\phi(aka^{-1})=aka^{-1}t^{-1}$. Then $W$ and $k$ commute as follows:
\begin{align*}
a^{-1}t^{-1}\phi(a)t\cdot k
&=a^{-1}t^{-1}\phi(ak)t\\
&=a^{-1}t^{-1}\phi(aka^{-1})\phi(a)t\\
&=k\cdot a^{-1}t^{-1}\phi(a)t
\end{align*}
Therefore, $\langle W, k\rangle\cong\mathbb{Z}\times\mathbb{Z}$ as required.
\end{proof}

We now prove Theorem \ref{thm:hyperbolicity}.

\begin{proof}[Proof of Theorem~\ref{thm:hyperbolicity}]
By assumption, $G=H\ast_{(K, \phi)}$ is hyperbolic and non-residually finite, and $K\lneq N_H(K)$.
Suppose that $K$ is infinite. Then $K$ is an infinite torsion group by Theorem \ref{thm:hyperbolicityZxZ}.
Now, as $K\leq G$ with $G$ hyperbolic, this is a contradiction \cite{gromov1987hyperbolic}. Hence, $K$ is finite.

Suppose that $H$ is residually finite. As $K$ is finite we have that $G$ is residually finite \cite[Theorem 3.1]{BaumslagTretkoff}, a contradiction. Hence, $H$ is non-residually finite.

Finally, note that $H$ is a quasi-convex subgroup of $G$ as $K$ and $\phi(K)$ are finite. Hence, $H$ is hyperbolic \cite[Proposition III.$\Gamma$.3.7]{bridson1999metric}.
\end{proof}

\bibliographystyle{amsalpha}
\bibliography{BibTexBibliography}
\end{document}